\documentclass{amsart}
\usepackage{amssymb}

\newtheorem{theorem}{Theorem}[section]
\newtheorem{lemma}[theorem]{Lemma}

\theoremstyle{definition}

\theoremstyle{remark}
\newtheorem{remark}[theorem]{Remark}

\begin{document}

\title[Solutions of the Differential Inequality with a~Null Lagrangian]
{Solutions of the Differential Inequality with a~Null Lagrangian:
Regularity and Removability of Singularities}

\author[A.~A.~Egorov]{A.~A.~Egorov$^1$}
\address{Sobolev Institute of Mathematics, Novosibirsk, Russia}
\email{yegorov@math.nsc.ru}

\begin{abstract} 
We prove a theorem on self-improving regularity for derivatives of 
solutions of the inequality $F(v'(x))\le KG(v'(x))$
constructed by means of a quasiconvex 
function~$F$
and a null 
Lagrangian~$G$.
We apply this theorem to improve the stability and H\"older 
regularity results of~\cite{Egor2008} and to establish a theorem
on removability of singularities for solutions of this inequality.
\end{abstract}

\maketitle

\section{Introduction} 

\addtocounter{footnote}{+1}\footnotetext{The 
author was supported by the Russian Foundation for Basic 
Research, the State Maintenance Program for the Leading Scientific 
Schools of the Russian Federation, and the Federal Target Program 
``Scientific and Educational Personnel of Innovation
Russia'' for 2009--2013 (contract No.~02.740.11.0457).}
In the present paper, 
which is a sequel 
to~\cite{Egor2008}, 
we study properties of solutions of the following inequality
\begin{equation}
\label{eq:nl-ineqlty}
F(v'(x))\le KG(v'(x))\quad\text{a.e.\ }V
\end{equation}
constructed by means of a quasiconvex 
function~$F$
and a null 
Lagrangian~$G$.
The results on closer of sets of such solutions with respect to 
the local convergence in the Lebesgue space, their H\"older regularity,
and precompactness of these sets with respect to the locally uniform 
convergence
\cite[Theorems~7 and~8 and Corollary~1]{Egor2008}
are applied to obtaining the stability theorems 
\cite[Theorems~1 and~3--6]{Egor2008}
for the class of solutions to the equation
\begin{equation}
\label{eq:nl-eqtn}
F(u'(x))=G(u'(x))\quad\text{a.e.\ }V.
\end{equation}
The main result below is the theorem on self-improving regularity
for derivatives of solutions 
of~\eqref{eq:nl-ineqlty} 
(Theorem~\ref{th:reglrty}).
We apply this result to improve the above-mentioned 
H\"older regularity and stability theorems
(see Theorems~\ref{th:Hldr}--\ref{th:W-stblty}).
Also we prove the theorem on removability of singularities for 
solutions of~\eqref{eq:nl-ineqlty}
(Theorem~\ref{th:remvblty}).  

Observe that if for a mapping
$v\colon V\subset\Bbb R^n\to\Bbb R^n$,
$n\ge2$,
we define 
$F(v'(x))=|v'(x)|^n$
and 
$G(v'(x))=\det v'(x)$
then 
inequality~\eqref{eq:nl-ineqlty}
is the dilatation inequality
\begin{equation}
\label{eq:diltn}
|v'(x)|^n\le K\det v'(x)\quad\text{a.e.\ }V.
\end{equation}
We remind that a solution of the class
$W^{1,n}(V;\Bbb R^n)$ 
of the dilatation inequality is called
{\it a mapping with 
$K$-bounded 
distortion or 
a
$K$-quasiregular mapping}.
The theory of mappings with bounded distortion is the key part of the
geometric function theory which has many diverse applications (for example, 
see monographs 
\cite{Iwan1992,IwanM1993,Kopy1990,Resh1982a,Resh1989,Resh1996}
and the bibliography therein).

A remarkable feature of the class of conformal mappings (mappings with 
$1$-bounded 
distortion)
is the stability phenomenon.
The first results on stability of classes of plane and spatial conformal 
mappings were obtained by M.~A.~Lavrent'ev while studying quasiconformal 
mappings (homeomorphic mappings with bounded 
distortion)~\cite{Lavr1935,Lavr1954}. 
Later, the theory of stability of conformal mappings which appeared in the 
framework of the theory of quasiconformal mappings was developed mainly by 
M.~A.~Lavrent'ev himself as well as P.~P.~Belinskii and Yu.~G.~Reshetnyak (for example, 
see the monographs 
\cite{Beli1974,Kopy1990,Resh1982a,Resh1989,Resh1996} 
and bibliography therein). 
One of the main results of this theory is the following assertion (for example, see
\cite{Beli1974,Kopy1990,Resh1982a,Resh1989,Resh1996,Rick1993}): 
{\it For a ball 
$B(x,r)\subset\Bbb R^n$, 
$n\ge2$, 
each 
$K$-quasiconformal mapping
$v\colon B(x,r)\to\Bbb R^n$ 
with 
coefficient~$K$ 
close to~1 deviates little in the 
$C$-norm 
from conformal mappings on each subball 
$B(x,\rho r)$, 
$0<\rho<1$; 
moreover, the deviation vanishes as 
$K\to1$.}
The stability property of conformal mappings is applied to obtaining important 
theorems both in the theory of quasiconformal mappings and its applications; 
therefore, finding other classes of mappings possessing the stability properties 
represents an interesting problem. 
Starting from the stability theory for conformal mappings, 
A.~P.~Kopylov~\cite{Kopy1982a} 
(also 
see~\cite{Kopy1990}) 
proposed the general conception of stability in the 
$C$-norm 
for classes of mappings, while he himself named 
$\xi$-stability. 
This conception agrees properly with the theory of stability of conformal mappings 
(see~\cite{Kopy1982a,Kopy1990}). 
Indeed, the above result is equivalent to the theorem on 
$\xi$-stability 
of the class of conformal mappings in the class of quasiconformal mappings (see  
\cite[Chapter 1, \S~1.3]{Kopy1990}). 
In the 
$\xi$-stability 
framework various stability theorems were obtained for classes of multidimensional 
holomorphic mappings, classes of solutions to elliptic systems of linear partial
differential equations, classes of homotheties, and a series of other mapping classes 
(for example, see the articles by 
Kopylov~\cite{Kopy1982a,Kopy1990,Kopy1995b}, 
Dairbekov~\cite{Dair1993,Dair1995b}, 
Sokolova~\cite{Soko1991a,Soko1991b}, and the bibliography therein).
Most of the above-mentioned
mapping classes can be considered as classes of solutions to equations of the 
form~\eqref{eq:nl-eqtn}.
In~\cite{Egor2008} we obtained a theorem on 
$\xi$-stability 
of classes of solutions 
to~\eqref{eq:nl-eqtn}
(see 
\cite[Theorem~1]{Egor2008}).
Some notes on the history of results on the self-improving regularity and on the 
removability of singularities for mappings with bounded distortion can be found
in the book of T.~Iwaniec and 
G.~Martin~\cite{IwanM2001}
(see also
\cite{Dair1992,FaraZ2006,IwanM1993,Iwan1992}).
We would like to point out that removability problems and regularity theory 
under minimal hypothesis are of crucial interest in PDE's.
The recent article of 
A.~P.~Kopylov~\cite{Kopy2007}
contains an exposition of new results on stability and regularity of solutions 
to elliptic systems of linear partial differential equations.
As 
in~\cite{Egor2003,Egor2005,Egor2007,Egor2008}
we develop the approaches and methods used for investigations
of mappings with bounded distortion to study properties of solutions 
of~\eqref{eq:nl-ineqlty}.
In particular, we apply the Hodge decomposition theory developed by 
T.~Iwaniec and G.~Martin~\cite{IwanM1993,Iwan1992,IwanM2001}
and used by them, for instance, for obtaining the theorems on 
self-improving regularity and removability of singularities for 
mappings with bounded distortion (for example, see
\cite[Theorem~14.4.1 and~17.3.1]{IwanM2001}).

We now describe the structure of the article. 
In~\S~2 we give the basic notation and terms. 
In~\S~3 we state the main results. 
In~\S~4 we expose the preliminary results.
The proof of Theorem~3.1 is presented in~\S~5. 
In \S~6 we give the proof of Theorem~3.5.

\section{Notation and Terminology}

Let 
$A$ 
be a set in 
$\Bbb R^n$. 
The topological boundary 
of~$A$ 
is denoted 
by~$\partial A$. 
The diameter 
of~$A$ 
is defined as 
$\operatorname{diam}A:=\sup\{|x-y|:x,y\in A\}$. 
The outer Lebesgue measure 
of~$A$ 
is denoted 
by~$|A|$.
We use the symbol
$\dim_HA$
for the Hausdorff dimension 
of~$A$.

The set
$\Bbb R^{m\times n}:=\{\zeta=(\zeta_{\mu\nu})
_{\genfrac{}{}{0pt}{}{\mu=1,\dots,m}{\nu=1,\dots,n}}
:\zeta_{\mu\nu}\in\Bbb R,\ \mu=1,\dots,m,\ \nu=1,\dots,n\}$
consists of all real
($m\times n$)-matrices. 
We identify a matrix
$\zeta=(\zeta_{\mu\nu})
_{\genfrac{}{}{0pt}{}{\mu=1,\dots,m}{\nu=1,\dots,n}}
\in\Bbb R^{m\times n}$
with the linear mapping
$(\zeta_1,\dots,\zeta_m)\colon\Bbb R^n\to\Bbb R^m$, 
where
$\zeta_\mu(x):=\sum_{\nu=1}^n\zeta_{\mu\nu}x_\nu$,
$\mu=1,\dots,m$,
$x=(x_1,\dots,x_n)\in\Bbb R^n$.
The operator norm 
in~$\Bbb R^{m\times n}$
is defined as
$|\zeta|:=\sup\{|\zeta(x)|:x\in\Bbb R^n,\ |x|<1\}$; 
and the Hilbert--Schmidt norm is defined as
$\|\zeta\|:=
\left(\sum_{\mu=1}^m\sum_{\nu=1}^n\zeta_{\mu\nu}^2\right)^{1/2}$.
The number of 
$k$-tuples 
of ordered indices in
$\Gamma_n^k:=\{I=(i_1,\dots,i_k):1\le i_1<\dots<i_k\le n,\
i_\varkappa\in\{1,\dots,n\},\ \varkappa=1,\dots,k\}$
equals the binomial coefficient
$\binom nk:=\frac{n!}{k!(n-k)!}$.
Given
$x\in\Bbb R^n$
and
$I\in\Gamma_n^k$,
we put
$x_I:=(x_{i_1},\dots,x_{i_k})\in\Bbb R^k$.
For 
$I\in\Gamma_n^k$
we denote
$dx_I=dx_{i_1}\wedge\dots\wedge dx_{i_k}$.
We use a convection that 
$dx_I=1$
if
$k=0$.
The entries of the
$k$th 
associated matrix
$M_k(\zeta)
:=(\operatorname{det}_{JI}\zeta)_{J\in\Gamma_m^k,I\in\Gamma_n^k}
\in\Bbb R^{\binom{m}{k}\times\binom{n}{k}}$ 
for the matrix
$\zeta\in\Bbb R^{m\times n}$
are the
$k\times k$-minors
$\operatorname{det}_{JI}\zeta:=\operatorname{det}
\left(
\begin{smallmatrix}
\zeta_{j_1i_1}&\hdots&\zeta_{j_1i_k}\\
\vdots&\ddots&\vdots\\
\zeta_{j_ki_1}&\hdots&\zeta_{j_ki_k}
\end{smallmatrix}
\right)$.
Here and in the sequel we enumerate the entries of
$\Upsilon\in\Bbb R^{\binom{m}{k}\times\binom{n}{k}}$
by
lexicographically ordered 
$k$-tuples
$I\in\Gamma_n^k$
and
$J\in\Gamma_m^k$,
i.e.\
$\Upsilon=(\gamma_{JI})_{J\in\Gamma_m^k,I\in\Gamma_n^k}$.
We identify
$M_1(\zeta)$
with~$\zeta$. 

The Jacobian matrix of
$u=(u_1,\dots,u_m)\colon U\subset\Bbb R^n\to\Bbb R^m$
at a point
$x\in U$
is the matrix
$u'(x):=\bigl(\frac{\partial u_\mu}{\partial x_\nu}(x)\bigr)
_{\genfrac{}{}{0pt}{}{\mu=1,\dots,m}{\nu=1,\dots,n}}
$.
If
$I\in\Gamma_n^k$
and
$J\in\Gamma_m^k$
then 
$\frac{\partial u_J}{\partial x_I}(x)
=\frac{\partial(u_{j_1},\dots,u_{j_k})}
{\partial(x_{i_1},\dots,x_{i_k})}(x)
:=\operatorname{det}_{JI}u'(x)$.

Let
$\mathcal V$
be a real vector space. 
We say that a function
$\Phi\colon\mathcal V\to\Bbb R$
is
{\it positively homogeneous of degree}
$p\in\Bbb R$
if
$\Phi(tx)=t^p\Phi(x)$
for all
$t>0$
and
$x\in\mathcal V\setminus\{0\}$.
Following 
Ch.~B.~Morrey~\cite{Morr1966},
we say that a continuous function
$F\colon\Bbb R^{m\times n}\to\Bbb R$
is 
{\it quasiconvex}, 
if
\begin{equation}
\label{eq:qc}
|B(0,1)|F(\zeta)\le\int_{B(0,1)}F(\zeta+\varphi'(x))\,dx
\end{equation}
for all
$\varphi\in C_0^\infty(B(0,1);\Bbb R^m)$
and
$\zeta\in\Bbb R^{m\times n}$.
Let 
$p\ge1$. 
Following 
M.~A.~Sychev~\cite{Sych1998}, 
we say that a quasiconvex 
function~$F$ 
is {\it strictly 
$p$-quasiconvex} 
if, for 
$\zeta\in\Bbb R^{m\times n}$ 
and 
$\varepsilon,C>0$, 
there is 
$\delta=\delta(\zeta,\varepsilon,C)>0$ 
such that, for each mapping 
$\varphi\in C_0^\infty(B(0,1);\Bbb R^m)$ 
satisfying 
$\|\varphi'\|_{L^p(B(0,1);\Bbb R^{m\times n})}
\le C|B(0,1)|^{1/p}$,
the condition
$\int_{B(0,1)}F(\zeta+\varphi'(x))\,dx
\le|B(0,1)|(F(\zeta)+\delta)$
implies 
$|\{x\in B(0,1):|\varphi'(x)|\ge\varepsilon\}|
\le\varepsilon|B(0,1)|$.
Observe that in the mathematical literature the term strictly 
quasiconvexity is also used for another property (which is close 
but nonequivalent to ours) consisting in the fact that the strict 
inequality in the definition of 
quasiconvexity~\eqref{eq:qc} 
is valid for nonzero 
mappings$~\varphi$ 
(for example, 
see~\cite{KnopS1984}). 
In this article we use the term in the sense of M.~A.~Sychev's 
definition~\cite{Sych1998}. 
In the case 
$p>1$ 
the notion of strictly 
$p$-quasiconvexity 
for 
functions~$F$ 
of this article is equivalent to the notion of strictly closed 
$p$-quasiconvexity 
from J.~Kristensen's
article~\cite{Kris1994} 
which is defined in terms of the theory of gradient Young measures 
(see \cite[Proposition 3.4]{Kris1994}).
Observe that we can replace the ball
$B(0,1)$ 
in the definitions of quasiconvexity and strictly 
$p$-quasiconvexity
by an arbitrary bounded 
domain~$U$ 
with 
$|\partial U|=0$
(for example,
see~\cite{Mull1999}).
A function
$G\colon\Bbb R^{m\times n}\to\Bbb R$
is a 
{\it null Lagrangian}
if both 
functions~$G$ 
and~$-G$ 
are quasiconvex. 
The term ``null Lagrangian'' appeared due to the following fact: 
The Euler--Lagrange equation corresponding to the variational 
integral
$\int_UG(u'(x))\,dx$
with null 
Lagrangian~$G$
holds identically for all admissible deformations
$u\colon U\subset\Bbb R^n\to\Bbb R^m$
(see~\cite{Ball1977b} 
and also~\cite{IwanM2001,BallCO1981,Daco1982,Daco1989,Mull1999}).
The only the affine combinations of minors (called 
{\it quasiaffine functions}) 
are null Lagrangians 
\cite{Edel1969,Land1942} 
(also see
\cite{Ball1977a,Ball1977b,BallCO1981,Daco1982,Daco1989,IwanM2001,
Morr1966,Mull1999}); 
i.e.
\begin{equation}
\label{eq:reprsntn-nl}
G(\zeta)=\gamma_0
+\sum_{k=1}^{\min\{m,n\}}\sum_{J\in\Gamma_m^k,I\in\Gamma_n^k}
\gamma_{JI}\operatorname{det}_{JI}\zeta,\quad 
\zeta\in\Bbb R^{m\times n},
\end{equation}
for some
$\gamma_0,\gamma_{JI}\in\Bbb R$.

\section{Statement of the Main Results}
\label{sc:main-reslts}

Fix a number
$k\in\Bbb N$,
$2\le k\le\min\{n,m\}$.
Below we assume that continuous functions
$F\colon\Bbb R^{m\times n}\to\Bbb R$
and 
$G\colon\Bbb R^{m\times n}\to\Bbb R$
satisfy the following conditions:

(H1)
$F$
is a quasiconvex function;

(H2)
$G$
is a null Lagrangian;
 
(H3)
$F$
and~$G$
are positively homogeneous of 
degree~$k$; 

(H4)
$\sup\{K\ge0:F(\zeta)\ge KG(\zeta),\ 
\zeta\in\Bbb R^{m\times n}\}=1$;

(H5)
$c_F:=\inf\{F(\zeta):\zeta\in\Bbb R^{m\times n},\ 
|\zeta|=1\}>0$;

(H6)
$d_G:=\sup\bigl\{\sum_{J\in\Gamma_m^k,I\in\Gamma_n^k}
|\gamma_{JI}||x_I|^2:x\in\Bbb R^n,\ |x|=1\bigr\}<kc_F/(n-k)$
in the case
$k<n$.

Here the 
coefficients~$\gamma_{JI}$ 
are taken 
from~\eqref{eq:reprsntn-nl} 
for the null 
Lagrangian~$G$.
By~(H3), the 
representation~\eqref{eq:reprsntn-nl} 
for the null 
Lagrangian~$G$
consists only of 
($k\times k$)-minors; 
i.e.,
\begin{equation}
\label{eq:reprsntn-G}
G(\zeta)=\sum_{J\in\Gamma_m^k,I\in\Gamma_n^k}
\gamma_{JI}\operatorname{det}_{JI}\zeta,
\quad\zeta\in\Bbb R^{m\times n}.
\end{equation}
Since 
$F$ 
is continuous, (H3) implies the inequalities
\begin{equation}
\label{eq:l-u-bnds-F}
c_F|\zeta|^k\le F(\zeta)\le C_F|\zeta|^k,\quad
\zeta\in\Bbb R^{m\times n},
\end{equation}
with the 
constants~$c_F$ 
from (H5) and
$C_F:=\sup\{F(\zeta):\zeta\in\Bbb R^{m\times n},\ |\zeta|=1\}<\infty$.

\begin{theorem}[Self-improving regularity]
\label{th:reglrty} 
Suppose 
that~$F$ 
and~$G$ 
satisfy 
{\rm(H2)--(H5)}.
Let
$K{\ge}1$. 
Then there exist two numbers
$q(F,G,K)$
and
$p(F,G,K)$
with 
$1<q(F,G,K)<k<p(F,G,K)$ 
such that for a given exponent
$p>q(F,G,K)$
every mapping 
$v\in W_{\operatorname{loc}}^{1,p}(V;\Bbb R^m)$,
which is defined on an open set 
$V\subset\Bbb R^n$
and satisfies 
inequality~\eqref{eq:nl-ineqlty},
actually lies in 
$W_{\operatorname{loc}}^{1,s}(V;\Bbb R^m)$
far all
$s\in(q(F,G,K),p(F,G,K))$. 
Moreover, for each test function
$\varphi\in C_0^\infty(V)$ 
we have the Caccioppoli-type inequality
\begin{equation}
\label{eq:Caccioppoli-ineqlty}
\|\varphi v'\|_{L^s(V;\Bbb R^{m\times n})}
\le C(F,G,K,s)\|v\otimes\varphi'\|_{L^s(V;\Bbb R^{m\times n})}
\end{equation}
for some constant
$C(F,G,K,s)>0$.
\end{theorem}

The following 
Theorems~\ref{th:Hldr}--\ref{th:W-stblty} 
are a straightforward consequence of 
Theorems~\ref{th:reglrty}
and 
\cite[Theorem~8, 4, and~6]{Egor2008}.

\begin{theorem}[H\"older regularity]
\label{th:Hldr}
Let 
$F$ 
and 
$G$ 
be functions satisfying {\rm (H2)--(H6)}. 
Put 
$K_0=\infty$ 
for 
$k=n$ 
and 
$K_0=\frac{kc_F}{(n-k)d_G}$
for
$k<n$. 
Suppose that 
$K\in[1,K_0)$ 
and 
$\delta\in(0,1)$ 
satisfy the inequality
\begin{equation}
\label{eq:inqlty-delta}
\frac{Kd_G}{kc_F}\le\frac1{n-k+k\delta}.
\end{equation}
Let 
$V$ 
be an open set 
in~$\Bbb R^n$. 
Then each solution
$v\in W_{\operatorname{loc}}^{1,p}(V;\Bbb R^m)$
of
inequality~\eqref{eq:nl-ineqlty}
satisfies the H\"older condition with 
exponent~$\delta$ 
on each compact subset 
in~$V$.
\end{theorem}

\begin{theorem}[Stability in the $C$-norm]
\label{th:C-stblty} 
Suppose 
that~$F$ 
and~$G$ 
satisfy 
{\rm(H1)--(H6)}.
Let
$K\ge1$,
and let
$q(F,G,K)$
denote the exponent from
Theorem~{\rm\ref{th:reglrty}}.
Let 
$V$ 
be a domain 
in~$\Bbb R^n$, 
and let 
$U$ 
be a compact subset 
in~$V$. 
Then there is a function 
$\alpha(K)=\alpha_{F,G,V,U}(K)$ 
defined for 
$1\le K<K_0$ 
and such that 
$\lim_{K\to1}\alpha(K)=\alpha(1)=0$ 
and, for each mapping 
$v\in W_{\operatorname{loc}}^{1,p}(V;\Bbb R^m)$,
$p>q(F,G,K)$,
which satisfies 
inequality~\eqref{eq:nl-ineqlty}
there is a mapping
$u\in W_{\operatorname{loc}}^{1,k}(V;\Bbb R^m)$
which is a solution
to~\eqref{eq:nl-eqtn}
such that
\begin{equation}
\label{eq:C-estmt-subdmn}
\|v-u\|_{C(U;\Bbb R^m)}\le\alpha(K)\operatorname{diam}v(V).
\end{equation}
\end{theorem}

The next theorem improves 
Theorems~\ref{th:C-stblty}
in the case when the 
function~$F$ 
satisfies the following condition:

(H1$'$) 
$F$ 
is strictly 
$k$-quasiconvex.

Note that 
condition~(H1$'$) 
is stronger that (H1). 
In this case, in addition to the 
estimate~\eqref{eq:C-estmt-subdmn}
of proximity (in the 
$C$-norm) 
of solutions of 
inequality~\eqref{eq:nl-ineqlty}
to solutions to 
equation~\eqref{eq:nl-eqtn}, 
we obtain proximity estimates (in the 
$L^k$-norm) 
for the derivatives of these mappings.

\begin{theorem}[Stability in the Sobolev norm]
\label{th:W-stblty} 
Suppose 
that~$F$ 
and~$G$ 
satisfy 
{\rm(H1$'$)}
and
{\rm(H2)--(H6)}.
Then the conclusion of 
Theorem~{\rm\ref{th:C-stblty}}
is valid together 
with~\eqref{eq:C-estmt-subdmn}
and the following inequality:
\begin{equation}
\label{eq:W-estmt-subdmn}
\|v'-u'\|_{W^{1,k}(U;\Bbb R^{m\times})}\le\alpha(K)\operatorname{diam}v(V).
\end{equation}
\end{theorem}

\begin{theorem}[Removability of singularities]
\label{th:remvblty}
Suppose 
that~$F$ 
and~$G$ 
satisfy 
{\rm(H2)--(H5)}.
Let
$K\ge1$,
and let
$q(F,G,K)$
denote the exponent from
Theorem~{\rm\ref{th:reglrty}}.
Consider a domain
$V\subset\Bbb R^n$.
Then for a closed 
subset~$E$
of~$V$ 
with the Hausdorff dimension
$\dim_H(E)<n-q(F,G,K)$
every bounded mapping 
$v\in W_{\operatorname{loc}}^{1,k}(V\setminus E;\Bbb R^m)$
which satisfies 
inequality~\eqref{eq:nl-ineqlty}
can be extended to a mapping of the class
$W_{\operatorname{loc}}^{1,k}(V;\Bbb R^m)$
which is defined over the whole
domain~$V$
and also satisfies inequality~\eqref{eq:nl-ineqlty}.
\end{theorem}

\section{Preliminary Results}

Let
$l\in\Bbb Z$
with
$0\le l\le n$,
and let
$p\ge1$.
Denote by
$L^p(\Bbb R^n;\Lambda^l)$ 
the space of differential 
$l$-forms 
on~$\Bbb R^n$ 
with coefficients in
$L^p(\Bbb R^n)$.

The following theorem is a modification of the result of 
T.~Iwaniec and G.~Martin on integral estimates concerning 
wedge products of closed differential forms 
\cite[Theorem~13.6.1]{IwanM2001}.

\begin{theorem}[Estimates beyond the natural exponent]
\label{th:estmts-beyond-ntrl-exp}
Let 
$n,k\in\Bbb N$
with
$2\le k\le n$.
Consider
$p_1,\dots,p_k,\varepsilon_1,\dots,\varepsilon_k\in\Bbb R$
and
$l_1,\dots,l_k\in\Bbb N$
such that
$1<p_\varkappa<\infty$,
$\frac1{p_1}+\dots+\frac1{p_k}=1$,
$-1\le2\varepsilon_\varkappa\le\frac{p_\varkappa-1}{p_\varkappa}$,
and
$\hat l:=n-l_1-\dots-l_k\ge0$.
Let
$\hat I=(\hat i_1,\dots,\hat i_{\hat l})\in\Gamma_n^{\hat l}$.
Suppose that
$(\varphi_1,\dots,\varphi_k)$
be 
$k$-tuple 
of closed differential forms with
$\varphi_\varkappa\in L^{(1-\varepsilon_\varkappa)p_\varkappa}
(\Bbb R^n;\Lambda^{l_\varkappa})$.
Then
\begin{multline}
\label{eq:estmts-beyond-ntrl-exp}
\int\frac{\varphi_1\wedge\dots\wedge\varphi_k\wedge dx_{\hat I}}
{|\varphi_1|^{\varepsilon_1}\dots|\varphi_k|^{\varepsilon_k}}
\\
\le C(p_1,\dots,p_k)\max(|\varepsilon_1|,\dots,|\varepsilon_k|)
\|\varphi_1\|_{L^{(1-\varepsilon_1)p_1}(\Bbb R^n;\Lambda^{l_1})}^{1-\varepsilon_1}
\dots
\|\varphi_k\|_{L^{(1-\varepsilon_k)p_k}(\Bbb R^n;\Lambda^{l_k})}^{1-\varepsilon_k}.
\end{multline}
\end{theorem}

\begin{remark}
For the case 
$\hat l=0$,
i.e.\
$dx_{\hat I}=1$,
the 
estimate~\eqref{eq:estmts-beyond-ntrl-exp}
was established in 
\cite[Theorem~13.6.1]{IwanM2001}.
In the proof of
Theorem~\ref{th:estmts-beyond-ntrl-exp}
we use the technique of Hodge decompositions developed 
in~\cite{IwanM2001}
(see also
\cite{Iwan1992,IwanM1993})
and applied for proving of
\cite[Theorem~13.6.1]{IwanM2001}. 
\end{remark}

\begin{proof}[Proof of Theorem~\ref{th:estmts-beyond-ntrl-exp}]
Observe that
$(1-\varepsilon_\varkappa)p_\varkappa\ge\frac{p_\varkappa+1}2>1$,
$\varkappa=1,\dots,k$.
We have
$\frac{\varphi_\varkappa}{|\varphi_\varkappa|^{\varepsilon_\varkappa}}
\in L^{p_\varkappa}(\Bbb R^n;\Lambda^{l_\varkappa})$.
Denote by
$W^{1,p}(\Bbb R^n;\Lambda^l)$,
$0\le l\le n$,
$p\ge0$, 
the space of differential 
$l$-forms
on~$\Bbb R^n$
with coefficients in
$W^{1,p}(\Bbb R^n)$.
We can consider the following Hodge decomposition in
$L^{p_\varkappa}(\Bbb R^n;\Lambda^{l_\varkappa})$
(\cite[Theorem~6.1]{IwanM1993},
see also
\cite[\S~10.6]{IwanM2001}):
\begin{equation}
\label{eq:Hodge-dcmpstn}
\frac{\varphi_\varkappa}{|\varphi_\varkappa|^{\varepsilon_\varkappa}}
=d\alpha_\varkappa+d^*\beta_\varkappa
\end{equation}
with some 
$\alpha_\varkappa\in W^{1,p_\varkappa}(\Bbb R^n;\Lambda^{l_\varkappa-1})$
and
$\beta_\varkappa\in W^{1,p_\varkappa}(\Bbb R^n;\Lambda^{l_\varkappa+1})$.
Here 
$d$
is the exterior derivative,
and
$d^*$
is its formal adjoint, the coexterior derivative.
The forms
$d\alpha_\varkappa$
and
$d^*\beta_\varkappa$,
$\varkappa=1,\dots,k$,
are uniquely determined and can be expressed by means of 
the Hodge projection operators
$$
E\colon L^p(\Bbb R^n;\Lambda^l)
\to dW^{1,p}(\Bbb R^n;\Lambda^{l-1})
\quad
\text{and}
\quad
E^*\colon L^p(\Bbb R^n;\Lambda^l)
\to d^*W^{1,p}(\Bbb R^n;\Lambda^{l+1})
$$
defined 
by~\cite[\S~10.6, formulas~(10.71) and~(10.72)]{IwanM2001}
for 
$1<p<\infty$
and
$1\le l\le n-1$.
Namely we have
\begin{equation}
\label{eq:parts-Hodge-dcmpstn}
d\alpha_\varkappa=E\left(\frac{\varphi_\varkappa}
{|\varphi_\varkappa|^{\varepsilon_\varkappa}}\right)
\quad\text{and}\quad
d^*\beta_\varkappa=E^*\left(\frac{\varphi_\varkappa}
{|\varphi_\varkappa|^{\varepsilon_\varkappa}}\right).
\end{equation}
Applying~\cite[Theorem~6.1]{IwanM1993},
we get the following bound for exact term:
\begin{equation}
\label{eq:bound-exact-term}
\|d\alpha_\varkappa\|_{L^{p_\varkappa}(\Bbb R^n;\Lambda^{l_\varkappa})}
\le C_1(p_\varkappa)
\|\varphi_\varkappa\|
_{L^{(1-\varepsilon_\varkappa)p_\varkappa}(\Bbb R^n;\Lambda^{l_\varkappa})}
^{1-\varepsilon_\varkappa}.
\end{equation}
By~\cite[\S~10.6, formulas~(10.73) and~(10.74)]{IwanM2001}
we have
$\operatorname{Ker}E=\{\varphi\in L^p(\Bbb R^n;\Lambda^l):\ d^*\varphi=0\}$
and
$\operatorname{Ker}E^*=\{\varphi\in L^p(\Bbb R^n;\Lambda^l):\ d\varphi=0\}$
for 
$1<p<\infty$
and
$1\le l\le n-1$.
Then
$E^*(\varphi_\varkappa)=0$.
Therefore we can write 
$d^*\beta_\varkappa$
as a commutator
$$
d^*\beta_\varkappa=E^*\left(\frac{\varphi_\varkappa}
{|\varphi_\varkappa|^{\varepsilon_\varkappa}}\right)
-\frac{E^*(\varphi_\varkappa)}
{|E^*(\varphi_\varkappa)|^{\varepsilon_\varkappa}}.
$$
Applying~~\cite[Theorem~12.2.1]{IwanM1993} 
(see also
\cite[Theorems~8.1 and~8.2]{Iwan1992}), 
we obtain
\begin{equation}
\label{eq:bound-coexact-term}
\|d\beta_\varkappa\|_{L^{p_\varkappa}(\Bbb R^n;\Lambda^{l_\varkappa})}
\le C_2(p_\varkappa)|\varepsilon_\varkappa|
\|\varphi_\varkappa\|
_{L^{(1-\varepsilon_\varkappa)p_\varkappa}(\Bbb R^n;\Lambda^{l_\varkappa})}
^{1-\varepsilon_\varkappa}.
\end{equation}
Using~\eqref{eq:Hodge-dcmpstn},
we have
\begin{multline}
\label{eq:all-Hodge-dcmpstns}
\int\frac{\varphi_1\wedge\dots\wedge\varphi_k\wedge dx_{\hat I}}
{|\varphi_1|^{\varepsilon_1}\dots|\varphi_k|^{\varepsilon_k}}
=\int (d\alpha_1+d^*\beta_1)\wedge\dots
\wedge(d\alpha_k+d^*\beta_k)\wedge dx_{\hat I}
\\
=\int d\alpha_1\wedge\dots\wedge d\alpha_k\wedge dx_{\hat I}
+\int\mathcal B.
\end{multline}
Since 
$p_1,\dots,p_k$
represents a H\"older conjugate tuple, by Stokes' formula via an 
approximation argument we obtain
\begin{equation}
\label{eq:Stokes}
\int d\alpha_1\wedge\dots\wedge d\alpha_k\wedge dx_{\hat I}=0.
\end{equation}
The 
integrand~$\mathcal B$
is a sum of wedge products of the type
$\psi_1\wedge\dots\wedge\psi_k\wedge dx_{\hat I}$,
where
$\psi_\varkappa$
is either
$d\alpha_\varkappa$
or
$d^*\beta_\varkappa$ 
and at least one
$d^*\beta_\varkappa$ 
is always present, with at most
$2^k-1$
terms.
Combining H\"older inequality 
with~\eqref{eq:bound-exact-term}
and~\eqref{eq:bound-coexact-term},
we get
\begin{multline*}
\int\psi_1\wedge\dots\wedge\psi_k\wedge dx_{\hat I}
\le C_3(k)
\|\psi_1\|_{L^{p_1}(\Bbb R^n;\Lambda^{l_1})}
\dots
\|\psi_k\|_{L^{p_k}(\Bbb R^n;\Lambda^{l_k})}
\\
\le C_4(p_1,\dots,p_k)\varepsilon
\|\varphi_1\|_{L^{(1-\varepsilon_1)p_1}(\Bbb R^n;\Lambda^{l_1})}^{1-\varepsilon_1}
\dots
\|\varphi_k\|_{L^{(1-\varepsilon_k)p_k}(\Bbb R^n;\Lambda^{l_k})}^{1-\varepsilon_k}
\end{multline*}
with
$\varepsilon:=\max(|\varepsilon_1|,\dots,|\varepsilon_k|)$.
This 
with~\eqref{eq:all-Hodge-dcmpstns}
and~\eqref{eq:Stokes}
yields~\eqref{eq:estmts-beyond-ntrl-exp}.
\end{proof}

The following theorem is a modification of the results 
of T.~Iwaniec and G.~Martin on integral estimates for Jacobians 
\cite[Theorems~7.8.1 and~13.7.1]{IwanM2001}
and is a consequence of Theorem~\ref{th:estmts-beyond-ntrl-exp}.

\begin{theorem}[Fundamental inequality for subdeterminants]
\label{th:L-p-ineqlty-minors-Jacbn-mtrx} 
Let
$n,m,k\in\Bbb N$ 
with
$2\le k\le\min(m,n)$.
Then there exists a constant
$C(k)\ge1$
such that for every distribution
$v=(v_1,\dots,v_m)\in \mathcal D'(\Bbb R^n;\Bbb R^m)$
with 
$v'\in L^p(\Bbb R^n;\Bbb R^{m\times n})$,
$1\le p<\infty$,
and for every
$I=(i_1,\dots,i_k)\in\Gamma_m^k$,
$J=(j_1,\ldots,j_k)\in\Gamma_n^k$
we have the inequality
\begin{equation}
\label{eq:L-p-ineqlty-minors-Jacbn-mtrx}
\left|\int|v'|^{p-k}\frac{\partial v_J}{\partial x_I}\right|
\le C(k)\left|1-\frac pk\right|\int|v'|^p.
\end{equation}
\end{theorem}

\begin{remark}
For the case 
$k=n=m$
the 
estimate~\eqref{eq:L-p-ineqlty-minors-Jacbn-mtrx}
was established in
\cite[Theorems~7.8.1 and~13.7.1]{IwanM2001}.
\end{remark}

\begin{proof}[Proof of Theorem~\ref{th:L-p-ineqlty-minors-Jacbn-mtrx}]
Let
$p_\varkappa:=k$,
$\varepsilon_\varkappa:=\varepsilon:=1-\frac pk$,
and
$l_\varkappa:=1$
for
$\varkappa=1,\dots,k$.
Then
$1<p_\varkappa<\infty$,
$\frac1{p_1}+\dots+\frac1{p_k}=1$,
$\hat l:=n-k=n-l_1-\dots-l_k\ge0$,
$(1-\varepsilon_\varkappa)p_\varkappa=p$,
and
$\max(|\varepsilon_1|,\dots,|\varepsilon_k|)=|\varepsilon|
=\left|1-\frac pk\right|$.
Let
$\varphi_\varkappa:=dv_{j_\varkappa}
\in L^{(1-\varepsilon_\varkappa)p_\varkappa}
(\Bbb R^n;\Lambda^{l_\varkappa})$.
Let 
$\hat I=(\hat i_1,\dots,\hat i_{\hat l})\in\Gamma_n^{\hat l}$
be the ordered
$\hat l$-tuple
such that
$\{\hat i_1,\dots,\hat i_{\hat l}\}
=\{1,\dots,n\}\setminus\{i_1,\dots,i_k\}$.
We chose the sign
$\operatorname{sgn}I$
such that
$\operatorname{sgn}Idx_I\wedge dx_{\hat I}=dx_1\wedge\dots\wedge dx_n$.

When
$p$
lies outside the interval
$\left(\frac{k+1}2,\frac{3k}2\right)$
the estimate is clear 
as~\eqref{eq:L-p-ineqlty-minors-Jacbn-mtrx}
always holds with~1 in place
$C(k)\left|1-\frac pk\right|$.
In this case
$\left|1-\frac pk\right|\ge\frac{k-1}{2k}$
and 
inequality~\eqref{eq:L-p-ineqlty-minors-Jacbn-mtrx}
holds with
$C(k)=\frac{2k}{k-1}$.

Suppose that
$k+1\le2p\le3k$.
Then
$-1\le2\varepsilon_\varkappa\le\frac{p_\varkappa-1}{p_\varkappa}$
and
$|\varepsilon|\le1/2$.
Applying 
Theorem~\ref{th:estmts-beyond-ntrl-exp},
we obtain
\begin{multline}
\label{eq:estmts-beyond-ntrl-exp-minors-Jacbn-mtrx}
\int\frac{\frac{\partial v_J}{\partial x_I}}
{|dv_{j_1}|^\varepsilon\dots|dv_{j_k}|^\varepsilon}
=\int\frac{\operatorname{sgn}Idv_{j_1}\wedge\dots\wedge dv_{j_k}\wedge dx_{\hat I}}
{|dv_{j_1}|^{\varepsilon_1}\dots|dv_{j_k}|^{\varepsilon_k}}
\\
\le C_1(k)|\varepsilon|
\|dv_{j_1}\|_{L^p(\Bbb R^n;\Lambda^1)}^{1-\varepsilon}
\dots\|dv_{j_k}\|_{L^p(\Bbb R^n;\Lambda^1)}^{1-\varepsilon}
\le C_1(k)\varepsilon\int|v'|^p.
\end{multline}
Using the elementary inequalities
$\left|\frac{\partial v_J}{\partial x_I}\right|
\le |dv_{j_1}|\dots|dv_{j_k}|$
and
$|a-a^{1-\varepsilon}|\le|\varepsilon|$
for 
$0\le a\le1$
and
$-1<\varepsilon<1$,
we have
\begin{multline*}
\left|\frac{\frac{\partial v_J}{\partial x_I}}{|v'|^{\varepsilon k}}
-\frac{\frac{\partial v_J}{\partial x_I}}
{|dv_{j_1}|^\varepsilon\dots|dv_{j_k}|^\varepsilon}\right|
\\
=\frac{\left|\frac{\partial v_J}{\partial x_I}\right||v'|^p}
{|dv_{j_1}|\dots|dv_{j_k}|}
\left|\frac{|dv_{j_1}|\dots|dv_{j_k}|}{|v'|^k}
-\left(\frac{|dv_{j_1}|\dots|dv_{j_k}|}{|v'|^k}\right)^{1-\varepsilon}\right|
\le |\varepsilon||v'|^p.
\end{multline*}
Combining this 
with~\eqref{eq:estmts-beyond-ntrl-exp-minors-Jacbn-mtrx},
we obtain
\begin{multline*}
\left|\int|v'|^{p-k}\frac{\partial v_J}{\partial x_I}\right|
\\
\le\int\left|\frac{\frac{\partial v_J}{\partial x_I}}{|v'|^{\varepsilon k}}
-\frac{\frac{\partial v_J}{\partial x_I}}
{|dv_{j_1}|^\varepsilon\dots|dv_{j_k}|^\varepsilon}\right|
+\int\left|\frac{\frac{\partial v_J}{\partial x_I}}
{|dv_{j_1}|^\varepsilon\dots|dv_{j_k}|^\varepsilon}\right|
\le (C_1(k)+1)|\varepsilon|\int|v'|^p.
\end{multline*}
\end{proof}

In the proof of 
Theorem~\ref{th:reglrty}
we use the following version of Gehring's lemma 
(see, for example,~\cite[Corollary~14.3.1]{IwanM2001}):

\begin{lemma}[Gehring's Lemma]
\label{l:Gehring}
Suppose 
$f$
and
$g$
are non-negative functions of class
$L^q(\Bbb R^n)$,
$1<q<\infty$,
and satisfy
$$
\left(\frac1{|B(a,R)|}\int_Qf^q\right)^{1/q}
\le\frac A{B(a,2R)}\int_{B(a,2R)}f
+\left(\frac1{|B(a,2R)|}\int_{B(a,2R)}g^q\right)^{1/q}
$$
for all balls
$B(a,R)\subset\Bbb R^n$
and some constant
$A>0$.
Then there exists a new exponent 
$q'=q'(n,q,A)>p$
and a constant
$C=C(n,q,A)>0$
such that
$$
\int f^{q'}\le C\int g^{q'}.
$$
\end{lemma}

\section{Proof of the Self-improving Regularity Theorem}

We are now in a position to prove 
Theorem~\ref{th:reglrty} 
given in
Section~\ref{sc:main-reslts}.

\begin{proof}[Proof of Theorem~\ref{th:reglrty}]
Let
$p>1$.
Obviously, we may assume that
$\varphi\ge0$
as otherwise we could consider
$|\varphi|$
which has not effect on 
inequality~\eqref{eq:Caccioppoli-ineqlty}.
Consider the auxiliary mapping
$h:=\varphi v\in W^{1,p}(\Bbb R^n;\Bbb R^m)$.
We have
$h'=\varphi v'+v\otimes\varphi'$.
Using 
\eqref{eq:l-u-bnds-F}
and
\eqref{eq:nl-ineqlty},
we deduce
\begin{multline}
\label{eq:bnd-dervtv-h-k}
|h'|^k\le(|\varphi v'|+|v\otimes\varphi'|)^k
=\varphi^k|v'|^k
+\sum_{\varkappa=0}^{k-1}\binom k\varkappa
|\varphi v'|^\varkappa|v\otimes\varphi'|^{k-\varkappa}
\\
\le c_F^{-1}\varphi^kF(v')
+\sum_{\varkappa=0}^{k-1}\binom k\varkappa
(|h'|+|v\otimes\varphi'|)^\varkappa|v\otimes\varphi'|^{k-\varkappa}
\\
\le c_F^{-1}\varphi^kF(v')
+\sum_{\varkappa=0}^{k-1}\binom k\varkappa
(|h'|+|v\otimes\varphi'|)^\varkappa|v\otimes\varphi'|^{k-\varkappa}
\\
\le c_F^{-1}KG(\varphi v')
+C_1(k)(|h'|+|v\otimes\varphi'|)^{k-1}|v\otimes\varphi'|
\\
=c_F^{-1}KG(h'-v\otimes\varphi')
+C_1(k)(|h'|+|v\otimes\varphi'|)^{k-1}|v\otimes\varphi'|
\\
\le c_F^{-1}KG(h')
+C_2(F,G)K(|h'|+|v\otimes\varphi'|)^{k-1}|v\otimes\varphi'|.
\end{multline}
Multiplying this inequality by
$|h'|^{p-k}$,
after a little manipulation we obtain
\begin{equation}
\label{eq:bnd-dervtv-h-p}
|h'|^p\le c_F^{-1}K|h'|^{p-k}G(h')
+C_3(F,G,p)K(|h'|+|v\otimes\varphi'|)^{p-1}|v\otimes\varphi'|.
\end{equation}
We observe here that clearly
$(|h'|+|v\otimes\varphi'|)^{p-1}|v\otimes\varphi'|$
enjoys higher integrability than
$|h'|^p$.
Using 
\eqref{eq:reprsntn-G},
we have
$|h'|^{p-k}G(h')
=\sum_{J\in\Gamma_m^k,I\in\Gamma_n^k}
\gamma_{JI}|h'|^{p-k}\operatorname{det}_{JI}v'$.
Applying 
Theorem~\ref{th:L-p-ineqlty-minors-Jacbn-mtrx},
we obtain
$\int|h'|^{p-k}G(h')\le C_4(G)\left|1-\frac pk\right|\int|h'|^p$.
Combining this 
with~\eqref{eq:bnd-dervtv-h-p},
we get
\begin{equation}
\label{eq:bnd-int-dervtv-h}
\int|h'|^p\le 
\frac{C_4(G)K}{c_F}\left|1-\frac pk\right|\int|h'|^p
+C_3(F,G,p)K\int(|h'|+|v\otimes\varphi'|)^{p-1}|v\otimes\varphi'|.
\end{equation}

Put 
$q(F,G,K)=k\left(1-\frac{c_F}{C_4(G)K}\right)$
and
$p(F,G,K)=k\left(1+\frac{c_F}{C_4(G)K}\right)$.
Suppose now that
$p\in(q(F,G,K),p(F,G,K))$.
Then
$\frac{C_4(G)K}{c_F}\left|1-\frac pk\right|<1$.
In this case 
inequality~\eqref{eq:bnd-int-dervtv-h} 
can be expresed as
$$
\int|h'|^p\le 
\frac{C_3(F,G,p)K}{1-\frac{C_4(G)K}{c_F}\left|1-\frac pk\right|}
\int(|h'|+|v\otimes\varphi'|)^{p-1}|v\otimes\varphi'|.
$$
We have
\begin{multline*}
\int(|h'|+|v\otimes\varphi'|)^p
\le2^{p-1}\int(|h'|^p+|v\otimes\varphi'|^p)
\\
\le2^{p-1}\left( 
\frac{C_3(F,G,p)K}{1-\frac{C_4(G)K}{c_F}\left|1-\frac pk\right|}
\int(|h'|+|v\otimes\varphi'|)^{p-1}|v\otimes\varphi'|
+\int|v\otimes\varphi'|^p\right)
\\
\le C(F,G,K,p)\int(|h'|+|v\otimes\varphi'|)^{p-1}|v\otimes\varphi'|
\\
\le C(F,G,K,p)\left[\int(|h'|+|v\otimes\varphi'|)^p\right]^{\frac{p-1}p}
\left[\int|v\otimes\varphi'|^p\right]^{\frac1p}.
\end{multline*}
Hence
$\||h'|+|v\otimes\varphi'|\|_{L^p(\Bbb R^n)}
\le C(F,G,K,p)\|v\otimes\varphi'\|_{L^p(\Bbb R^n;\Bbb R^{m\times n})}$.
Then, in view of the simple fact that
$|\varphi v'|\le|h'|+|v\otimes\varphi'|$,
we obtain the Cacciappoli-type estimate
\begin{equation}
\label{eq:Caccioppoli-ineqlty-p}
\|\varphi v'\|_{L^p(\Bbb R^n;\Bbb R^{m\times n})}
\le C(F,G,K,p)\|v\otimes\varphi'\|_{L^p(\Bbb R^n;\Bbb R^{m\times n})}.
\end{equation}
Of course now we observe that this inequality holds 
with~$p$
replaced 
by~$s$
for 
any~$s\in(q(F,G,K),p(F,G,K))$, 
provided we know {\it a priori} that
$v\in W_{\operatorname{loc}}^{1,s}(V;\Bbb R^m)$.

Let 
$S=\{s\in(q(F,G,K),p(F,G,K)): 
v\in W_{\operatorname{loc}}^{1,s}(V;\Bbb R^m)\}$.
We have
$p\in S$.
Therefore,
$S\ne\varnothing$.
For 
$s\in S$
we 
have~\eqref{eq:Caccioppoli-ineqlty};
the constant
$C(F,G,K,p)$
which depends continuously 
on~$s$
is finite in the range
$q(F,G,K)<s<p(F,G,K)$
but may blow up at the endpoints. 
This shows that 
$S$
is relatively closed in
$(q(F,G,K),p(F,G,K))$.
The theorem will be proved if we can show that
$S$
is open.
Certainly, if
$s\in S$,
then
$(q(F,G,K),s]\subset(q(F,G,K),p(F,G,K))$.
We are therefore left only with the task of showing higher 
integrability of the differential.
It is at this point that Gehring's lemma comes to the rescue.
We easily derive 
from~\eqref{eq:bnd-dervtv-h-k}
reverse H\"older inequality for
$h'$.
Let
$B_R:=B(a,R)\subset B(a,2R)=:B_{2R}$
be a concentric balls
in~$V$
and let
$0\le\eta\le1$
be a function in
$C_0^\infty(B_{2R})$
which is equal 
to~$1$
on~$B_R$
and has
$|\eta'|\le\frac{C(n)}R$.
Now we repeat the above calculations with some modifications to 
obtain the Cacciapolli-type estimate 
for~$h-h_{B_{2R}}$
and~$\eta$,
where 
$h_{B_{2R}}:=\frac1{|B_{2R}|}\int_{B_{2R}}|h|$.
Consider the mapping
$H:=\eta(h-h_{B_{2R}})$.
We have
$H'=\eta h'+(h-h_{B_{2R}})\otimes\eta'$.
Using 
\eqref{eq:bnd-dervtv-h-k},
we deduce
\begin{multline}
\label{eq:bnd-dervtv-H-k}
|H'|^k\le(|\eta h'+(h-h_{B_{2R}})\otimes\eta'|)^k
\\
=\eta^k|h'|^k
+\sum_{\varkappa=0}^{k-1}\binom k\varkappa
|\eta h'|^\varkappa|(h-h_{B_{2R}})\otimes\eta'|^{k-\varkappa}
\\
\le c_F^{-1}KG(\eta h')
+C_2(F,G)K(|\eta h'|+\eta|v\otimes\varphi'|)^{k-1}\eta|v\otimes\varphi'|
\\
+\sum_{\varkappa=0}^{k-1}\binom k\varkappa
(|\eta h'|)^\varkappa|(h-h_{B_{2R}})\otimes\eta'|^{k-\varkappa}
\le c_F^{-1}KG(H'-(h-h_{B_{2R}})\otimes\eta')
\\
+C_2(F,G)K(|H'|+|(h-h_{B_{2R}})\otimes\eta'|+\eta|v\otimes\varphi'|)^{k-1}
\eta|v\otimes\varphi'|
\\
+\sum_{\varkappa=0}^{k-1}\binom k\varkappa
(|H'|+|(h-h_{B_{2R}})\otimes\eta'|)^\varkappa|(h-h_{B_{2R}})\otimes\eta'|^{k-\varkappa}
\le c_F^{-1}KG(H')
\\
+C_5(F,G)K(|H'|+|(h-h_{B_{2R}})\otimes\eta'|+\eta|v\otimes\varphi'|)^{k-1}
(|(h-h_{B_{2R}})\otimes\eta'|+\eta|v\otimes\varphi'|).
\end{multline}
Multiplying this inequality by
$|H'|^{p-n}$,
after a little manipulation we obtain
\begin{multline}
\label{eq:bnd-dervtv-H-p}
|H'|^p\le c_F^{-1}K|H'|^{p-n}G(H')
\\
+C_6(F,G,p)K(|H'|+|(h-h_{B_{2R}})\otimes\eta'|+\eta|v\otimes\varphi'|)^{k-1}
(|(h-h_{B_{2R}})\otimes\eta'|+\eta|v\otimes\varphi'|).
\end{multline}
Using 
again~\eqref{eq:reprsntn-G}
and 
Theorem~\ref{th:L-p-ineqlty-minors-Jacbn-mtrx},
we obtain
\begin{multline*}
\int|H'|^p\le 
\frac{C_6(F,G,p)K}{1-\frac{C_4(G)K}{c_F}\left|1-\frac pk\right|}\times
\\
\times\int(|H'|+|(h-h_{B_{2R}})\otimes\eta'|+\eta|v\otimes\varphi'|)^{p-1}
(|(h-h_{B_{2R}})\otimes\eta'|+\eta|v\otimes\varphi'|).
\end{multline*}
We have
\begin{multline*}
\int(|H'|+|(h-h_{B_{2R}})\otimes\eta'|+\eta|v\otimes\varphi'|)^p
\\
\le2^{p-1}\int(|H'|^p+(|(h-h_{B_{2R}})\otimes\eta'|+\eta|v\otimes\varphi'|)^p)
\le2^{p-1}\left( 
\frac{C_6(F,G,p)K}{1-\frac{C_4(G)K}{c_F}\left|1-\frac pk\right|}\right.\times
\\
\times
\int(|H'|+|(h-h_{B_{2R}})\otimes\eta'|+\eta|v\otimes\varphi'|)^{p-1}
(|(h-h_{B_{2R}})\otimes\eta'|+\eta|v\otimes\varphi'|)
\\
\left.+\int(|(h-h_{B_{2R}})\otimes\eta'|+\eta|v\otimes\varphi'|)^p\right)
\\
\le C_7(F,G,K,p)\int(|H'|+|(h-h_{B_{2R}})\otimes\eta'|+\eta|v\otimes\varphi'|)^{p-1}
(|(h-h_{B_{2R}})\otimes\eta'|+\eta|v\otimes\varphi'|)
\\
\le C_7(F,G,K,p)\left[
\int(|H'|+|(h-h_{B_{2R}})\otimes\eta'|+\eta|v\otimes\varphi'|)^p
\right]^{\frac{p-1}p}\times
\\
\times
\left[\int(|(h-h_{B_{2R}})\otimes\eta'|+\eta|v\otimes\varphi'|)^p\right]^{\frac1p}.
\end{multline*}
Hence
\begin{multline*}
\||H'|+|(h-h_{B_{2R}})\otimes\eta'|+\eta|v\otimes\varphi'|\|_{L^p(\Bbb R^n)}
\\
\le C_7(F,G,K,p)
\||(h-h_{B_{2R}})\otimes\eta'|+\eta|v\otimes\varphi'|\|_{L^p(\Bbb R^n)}.
\end{multline*}
Then, in view of the simple facts that
$$
|\eta h'|\le|H'|+|(h-h_{B_{2R}})\otimes\eta'|
\le|H'|+|(h-h_{B_{2R}})\otimes\eta'|+\eta|v\otimes\varphi'|
$$
and
$$
(|(h-h_{B_{2R}})\otimes\eta'|+\eta|v\otimes\varphi'|)^p
\le|(h-h_{B_{2R}})\otimes\eta'|^p+\eta|v\otimes\varphi'|^p,
$$
we obtain the Cacciappoli-type estimate
$$
\int|\eta h'|^p
\le C_8(F,G,K,p)\int|(h-h_{B_{2R}})\otimes\eta'|^p
+C_8(F,G,K,p)\int\eta|v\otimes\varphi'|^p.
$$
Using the properties of the test 
function~$\eta$,
we get
$$
\int_{B_R}|h'|^p
\le C_9(F,G,K,p)R^{-p}\int_{B_{2R}}|h-h_{B_{2R}}|^p
+C_9(F,G,K,p)\int_{B_{2R}}|v\otimes\varphi'|^p.
$$
Combining this with the Poincar\'e--Sobolev inequality (see, for example,
\cite[Theorem~4.10.3]{IwanM2001}),
we obtain
\begin{multline*}
\frac1{|B_R|}\int_{B_R}|h'|^p
\le C_{10}(F,G,K,p)\left(\frac1{B_{2R}}\int_{B_{2R}}|h'|^{\frac{np}{n+p}}\right)^{\frac{n+p}n}
\\
+\frac{C_{10}(F,G,K,p)}{B_{2R}}\int_{B_{2R}}|v\otimes\varphi'|^p.
\end{multline*}
Hence
\begin{multline*}
\left(\frac1{|B_R|}\int_{B_R}|h'|^p\right)^{\frac n{n+p}}
\le \frac{C_{11}(F,G,K,p)}{B_{2R}}\int_{B_{2R}}|h'|^{\frac{np}{n+p}}
\\
+\left(\frac1{|B_{2R}|}\int_{B_{2R}}(C_12(F,G,K,p)|v\otimes\varphi'|)^p\right)^{\frac n{n+p}}.
\end{multline*}
Put
$q=\frac{n+p}n>1$,
$f=|h'|^{\frac{np}{n+p}}$,
and
$g=|v\otimes\varphi'|^{\frac{np}{n+p}}$.
By 
Lemma~\ref{l:Gehring}
we conclude that
$f$
is integrable with a power slightly larger 
than~$q$.
This in turn means that
$h'$
is integrable with a slightly higher power 
then~$p$
and so
$v\in W_{\operatorname{loc}}^{1,p'}(V;\Bbb R^m)$
for some
$p'>p$.
\end{proof}

\section{Proof of the Removability Theorem}

As in the proof 
of~\cite[Theorem~17.3.1]{IwanM2001}
there are two key components in the proof of 
Theorem~\ref{th:remvblty}.
Firstly, the assumption on the size of 
set~$E$
implies that
$E$
has zero 
$s$-capacity
for an appropriate value 
of~$s$.
Secondly, the Cacciappoli 
estimate~\eqref{eq:Caccioppoli-ineqlty} 
holds for this particular value
of~$s$.

\begin{proof}[Proof of Theorem~\ref{th:remvblty}]
We have
$q(F,G,K)<n-\dim_H(E)$.
Let
$$
s\in(q(F,G,K),n-\dim_H(E)).
$$
From~\cite[Theorem~17.2.1]{IwanM2001} 
(see also~\cite{BojaI1983,Mazy1985,Rick1993,Vais1975}),
we obtain that the set
$E$
has zero 
$s$-capacity.
It is clear that
$|E|=0$.
Further,
Theorem~\ref{th:reglrty}
gives the Caccioppoli estimate
\begin{equation}
\label{eq:Caccioppoli-u}
\|\varphi v'\|_{L^s(V;\Bbb R^{m\times n})}
\le C\|v\otimes\varphi'\|_{L^s(V;\Bbb R^{m\times n})}
\end{equation}
for every
$\varphi\in C_0^\infty(V\setminus E)$,
where the constant 
$C=C(F,G,K,s)$
does not depend on the test 
function~$\varphi$
or the 
function~$v$.

Let
$\chi\in C_0^\infty(V)$
and
$E':=E\cap\operatorname{supp}\chi$.
Then
$E'$
has zero 
$s$-capacity.
Therefore there exists a sequence of functions
$(\eta_j\in C_0^\infty(V))_{j\in\Bbb N}$
such that
$0\le\eta_j\le1$;
$\eta_j=1$
on some neighbourhood
of~$E'$,
$\lim_{j\to\infty}\eta_j=0$
almost everywhere 
in~$\Bbb R^n$,
and
$\lim_{j\to\infty}\int|\eta'_j|^s=0$.
Put
$\varphi_j:=(1-\eta_j)\chi\in C_0^\infty(V\setminus E)$
and
$v_j:=\varphi_j v\in W_0^{1,s}(V;\Bbb R^m)$.
Then the 
mappings~$v_j$
are bounded in
$L^\infty(V;\Bbb R^m)$
converge to
$\chi v$
almost everywhere.
We have
$v'_j=\varphi_j v'+v\otimes \varphi'_j$
and
$\varphi'_j=-\chi\eta'_j+(1-\eta_j)\chi'$.
Using~\eqref{eq:Caccioppoli-u},
we obtain
\begin{multline*}
\|v'_j\|_{L^s(V;\Bbb R^{m\times n})}
\le\|\varphi_j v'\|_{L^s(V;\Bbb R^{m\times n})}
+\|v\otimes\varphi'_j\|_{L^s(V;\Bbb R^{m\times n})}
\\
\le(1+C)\|v\otimes \varphi'_j\|_{L^s(V;\Bbb R^{m\times n})}
\\
\le(1+C)\left(\|\chi\|_{L^\infty(V)}\|v\|_{L^\infty(V;\Bbb R^m)}
\|\eta'_j\|_{L^s(V;\Bbb R^n)}
+\|(1-\eta_j)v\otimes\chi'\|_{L^s(V;\Bbb R^{m\times n})}\right).
\end{multline*}
Passing to the limit 
over~$j$,
we get
\begin{equation}
\label{eq:bnd-dervtv-u-j}
\limsup_{j\to\infty}\|v'_j\|_{L^s(V;\Bbb R^{m\times n})}
\le(1+C)\|v\otimes\chi'\|_{L^s(V;\Bbb R^{m\times n})}.
\end{equation}
Therefore the sequence
$(v_j)_{j\in\Bbb N}$
is bounded in
$W^{1,s}(V;\Bbb R^m)$.
Hence there exists its subsequence 
$(v_{j_s})_{s\in\Bbb N}$
converges weakly in 
$W^{1,s}(V;\Bbb R^m)$
to a mapping in this Sobolev space.
Clearly, this limit coincides with
$\chi v$
almost everywhere 
in~$V$.

Therefore
$\chi v\in W_0^{1,s}(V;\Bbb R^m)$
for all test functions
$\chi\in C_0^\infty(V)$.
This yields
$v\in W_{\operatorname{loc}}^{1,s}(V;\Bbb R^m)$.
Since
$v$
is a solution of 
inequality~\eqref{eq:nl-ineqlty}
almost everywhere 
in~$V$,
Theorem~\ref{th:reglrty} 
yields 
$v\in W_{\operatorname{loc}}^{1,k}(V;\Bbb R^m)$.
\end{proof}

\bibliographystyle{amsunsrt}

\end{document}